\documentclass[12pt]{amsart}
\usepackage{amsmath,amsthm,amsfonts,amssymb,latexsym,enumerate,xcolor,comment}
\usepackage[english]{babel}

\newcommand{\bN}{{\mathbf N}}

\newcommand{\bF}{{\mathbf F}}
\newcommand{\bO}{{\mathbf O}}
\newcommand{\bC}{{\mathbf C}}

\newcommand{\bZ}{{\mathbf Z}}

\newcommand{\Aut}{{{\operatorname{Aut}}}}

\newcommand{\Pro}{{{\operatorname{Pr}}}}

\newcommand{\GL}{\operatorname{GL}}

\newcommand{\SL}{\operatorname{SL}}

\newcommand{\Sym}{\operatorname{Sym}}
\newcommand{\Syl}{\operatorname{Syl}}

\newcommand{\fpr}{\operatorname{fpr}}

\newtheorem{thm}{Theorem}[section]
\newtheorem{lem}[thm]{Lemma}

\newtheorem{prop}[thm]{Proposition}
\newtheorem{cor}[thm]{Corollary}

\newtheorem*{thmA}{Theorem A}
\newtheorem*{conA'}{Conjecture A'}
\newtheorem*{thmB}{Theorem B}
\newtheorem*{thmC}{Theorem C}
\newtheorem*{conD}{Conjecture D}

\newtheorem*{thmE}{Theorem E}
\newtheorem*{thmF}{Theorem F}
\newtheorem*{thmG}{Theorem G}
\newtheorem*{thmH}{Theorem H}

\theoremstyle{definition}
\newtheorem{rem}[thm]{Remark}

\newtheorem*{note}{Note}

\numberwithin{equation}{section}

\def\syl#1#2{{\rm Syl}_{#1}(#2)}
\def\norm#1#2{{\bf N}_{#1}(#2)}

\begin{document}

\title[Fixed point ratios, Sylow Numbers and coverings of $p$-elements]{Fixed point ratios, Sylow Numbers and coverings of $p$-elements  in finite groups }

\author[R. M. Guralnick]{Robert M. Guralnick}
\address{Department of Mathematics, University of Southern California, Los Angeles, CA 90089-2532, USA}
\email{guralnic@usc.edu}

\author[A. Mar\'oti]{Attila Mar\'oti}
\address{Hun-Ren Alfr\'ed R\'enyi Institute of Mathematics, Re\'altanoda Utca 13-15, H-1053, Budapest, Hungary}
\email{maroti@renyi.hu}

\author[J. Mart\'{\i}nez]{Juan Mart\'{\i}nez Madrid}
\address{Departament de Matem\`atiques, Universitat de Val\`encia, 46100
  Burjassot, Val\`encia, Spain}
\email{Juan.Martinez-Madrid@uv.es}

\author[A. Moret\'o]{Alexander Moret\'o}
\address{Departament de Matem\`atiques, Universitat de Val\`encia, 46100
  Burjassot, Val\`encia, Spain}
\email{alexander.moreto@uv.es}

\author[N. Rizo]{Noelia Rizo}
\address{Departament de Matem\`atiques, Universitat de Val\`encia, 46100
  Burjassot, Val\`encia, Spain}
\email{noelia.rizo@uv.es}

\thanks{The first author was partially supported by the NSF grant DMS-1901595 and a Simons Foundation Fellowship 609771. The second, the third, the fourth and the fifth authors were supported by Ministerio de Ciencia e Innovaci\'on (Grants PID2019-103854GB-I00 and PID2022-137612NB-I00 funded by MCIN/AEI/10.13039/501100011033 and ``ERDF A way of making Europe") and CIAICO/2021/163. The second author was also supported by the National Research, Development and Innovation Office (NKFIH) Grant No.~K138596, No.~K132951 and Grant No.~K138828. The third author was also funded by CIACIF/2021/228 funded by Generalitat Valenciana. The fourth and  fifth author were supported by a CDEIGENT grant CIDEIG/2022/29 funded by Generalitat Valenciana.}



\date{\today}

\begin{abstract}
Fixed point ratios for primitive permutation groups have been extensively studied. Relying on a recent work of Burness and Guralnick, we obtain further results in the area. For a prime $p$ and a finite group $G$, we use fixed point ratios to study the number of Sylow $p$-subgroups of $G$ and the minimal size of a covering by proper subgroups of the set of $p$-elements of $G$.
\end{abstract}

\dedicatory{Dedicated to Gabriel Navarro on his 60th birthday.}

\maketitle



\section{Introduction}

Let $G$ be a transitive permutation group acting on a finite set $\Omega$ with point stabilizer $H$. The fixed point ratio of an element $x \in G$ on $\Omega$ is defined to be
$$
\fpr(x,\Omega)=\frac{|\bC_{\Omega}(x)|}{|\Omega|},
$$
where $\bC_{\Omega}(x)=\{\alpha\in\Omega\mid \alpha^x=\alpha\}$ is the set of fixed points of $x$.
This is equal to $|x^G\cap H|/|x^G|$ where $x^G$ is the conjugacy class of $x$ in $G$. (See \cite[Lemma 1.2(iii)]{bur}.) 

Fixed point ratios have been extensively studied since the origins of group theory in the 19th century. They have a wide range of applications, both in group theory and in other areas of mathematics. For instance, they have been used to investigate the structure of monodromy groups of coverings of the Riemann sphere. This was used to prove a conjecture of Guralnick and Thompson \cite{gt}. We refer the reader to \cite{bur} and to the introduction of \cite{bg} for good overviews of the results and applications. 

Recently, new results on fixed point ratios of primitive permutation groups were obtained by Burness and Guralnick \cite{bg}. They proved that if $G$ is a primitive permutation group on a finite set $\Omega$ and $x\in G$ is an element of prime order $p$, then $\fpr(x,\Omega)\leq 1/(p+1)$, except for a small list of exceptions (see \cite[Theorem 1]{bg}). This has had several applications. First, Burness and Guralnick  classified the primitive permutation groups of degree $m$ with minimal degree less than $2m/3$, extending earlier work of Liebeck and Saxl \cite{LSax} and Guralnick and Magaard \cite{gur-mag}. Second, they used it to investigate the minimal index of a primitive permutation group, which is closely related to a question of Bhargava \cite{bar} on the Galois groups of polynomials. Third, in joint work with Moret\'o and Navarro \cite{bgmn}, it was used to study the commuting probability of $p$-elements (and more generally, the commuting probability of $\pi$-elements in \cite{mar}).

Let $p$ be a prime. Our first goal is to obtain a sharp upper bound for the fixed point ratio of some $p$-elements in primitive permutation groups whose orders are divisible by $p$. See Theorem \ref{oldA} and Corollary \ref{corprimitive}. This result has the following consequence for transitive permutation groups. 

\begin{thmA}
Let $G$ be a transitive permutation group on a finite set $\Omega$. Let $p$ be a prime divisor of the size of $G$. Let $P$ be a Sylow $p$-subgroup of $G$. 
\begin{itemize}
		
\item[(i)] If $G$ is generated by its $p$-elements, then there exists a $p$-element $x$ in $G$ such that $$\fpr(x,\Omega) \leq \frac{p-1}{2p-1}.$$ 

\item[(ii)] If $G$ is generated by its $p$-elements and if $G$ has no factor group isomorphic to a (non-abelian simple) alternating group $A_{m}$ with $p+1 < m < p^{2}-p$ or to the group $\mathrm{SL}(2,p+1)$ with $p$ a Mersenne prime, then there exists a $p$-element $x$ in $G$ such that $$\fpr(x,\Omega) \leq \frac{1}{p+1}.$$   

\item[(iii)] If $p>2$, then the number of orbits of $P$ on $\Omega$ is at most $$\frac{p}{2p-1} \cdot |\Omega|.$$ 

\item[(iv)] If $p>2$ and the subgroup generated by the $p$-elements of $G$ has no factor group isomorphic to an alternating group $A_{m}$ with $p+1 < m < p^{2}-p$, then 
\begin{itemize}
\item[(a)] the number of orbits of $P$ on $\Omega$ is at most $(2/(p+1)) |\Omega|$ and

\item[(b)] there exists a $p$-element $x$ in $G$ such that $\fpr(x,\Omega) < 2/(p+1)$.  
\end{itemize}
\end{itemize}
\end{thmA}

Our second goal is to obtain an asymptotic result.

\begin{thmB}
Let $G$ be a simple transitive permutation group acting on a finite set $\Omega$. Let $p$ be a prime divisor of the size of $G$. There exists a $p$-element $x\in G$  such that $\fpr(x,\Omega)\to 0$ as $|G|\to\infty$. 
\end{thmB}

We present two types of applications of Theorems A and B. Let $p$ be a prime. A classical topic in group theory is the study of Sylow numbers. The Sylow number $\nu_{p}(G)$ for a finite group $G$ is defined to be the number of Sylow $p$-subgroups of $G$. For any finite group $G$, we have $\nu_p(G)\equiv1\pmod{p}$ by Sylow's theorem. One of Brauer's problems from his famous list published in 1963 \cite{bra} asks for the possible values of $\nu_p(G)$ where $G$ runs through all the finite groups. Not much is known about this. We have $\nu_p(H)\leq\nu_p(G)$ by Lemma \ref{lessorequal}. It is surprising that not many stronger results are known. Navarro \cite{nav} proved that if $G$ is $p$-solvable, then $\nu_p(H)$ divides $\nu_p(G)$. Using this result, Mar\'oti, Mart\'{\i}nez and Moret\'o \cite{mmm} proved that if $G$ is a  $p$-solvable group generated by $p$-elements and $H<G$ contains a Sylow $p$-subgroup of $G$, then   $\nu_p(H)\leq \nu_p(G)/(p+1)$.   

Using Theorem A, we may obtain a generalization of this bound in the following. 

\begin{thmC}
Let $p$ be a prime. Let $G$ be a finite group generated by $p$-elements. Let $H$ be a proper subgroup of $G$ containing a Sylow $p$-subgroup of $G$. We have 
$$\nu_p(H)\leq\frac{p-1}{2p-1}\nu_p(G).$$ 
Moreover, if $G$ has no factor group isomorphic to a (non-abelian) alternating group $A_{m}$ with $p+1 < m < p^{2}-p$ or to the group $\mathrm{SL}(2,p+1)$ with $p$ a Mersenne prime, then $$\nu_p(H)\leq\frac{\nu_p(G)}{p+1}.$$ 
\end{thmC}

The first inequality in Theorem C is sharp. For if $p>2$, $H = A_{2p-2}$ and $G = A_{2p-1}$, then $\nu_p(H) = ((p-1)/(2p-1)) \nu_p(G)$. Equality also holds for $p=2$ when $G=S_3$ and $H$ has order $2$.

If $p>3$ is a Mersenne prime, $G=\SL(2,p+1)$ and $H$ is a Borel subgroup of $G$, then $\nu_p(H)=2\nu_p(G)/(p+2)$.  Note that the hypothesis in Theorem C that $H$ contains a Sylow $p$-subgroup of $G$ is definitely necessary. Consider for instance any Frobenius group $G$ with complement of order $p^2$ and let $H$ be a subgroup of index $p$. In this case, $\nu_p(H)=\nu_p(G)$. This example can be generalized. Let $P$ be a $p$-group acting on a $p'$-group $V$ and let $P_0$ be a subgroup of $P$ such that $C_V(P)=C_V(P_0)$. Put $G=VP$ and $H=VP_0$. Then we have $\nu_p(G)=\nu_p(H)$. We characterize the pairs $(G,H)$ where equality holds in Lemma \ref{lessorequal}. However, the following could be true.

\begin{conD} 
For each odd prime $p$, there exists a rational number $f(p)$ less than $1$ such that whenever $G$ is a finite group and $H$ is a subgroup of $G$ with $\nu_p(H)<\nu_p(G)$, then 
$$\nu_p(H)\leq f(p) \cdot \nu_p(G).$$
\end{conD} 

It follows from Theorem C that Conjecture D holds when $H$ contains a Sylow $p$-subgroup of $G$, even when $p=2$. Fixed point ratios do not seem enough to handle the general case. In fact,  the example $(G,H) = (\mathrm{SL}(2,2^{k}), D_{2(2^{k}-1)})$ shows that the hypothesis that $p$ is odd is necessary. Note that in this case $\nu_2(G)=2^{k}+1$ and $\nu_2(H)=2^{k}-1$. 
We describe the structure of a minimal counterexample to this conjecture in Theorem \ref{thm:minimalcounterConjD} below. 
We refer the reader to Section 6 for a discussion on possible values of $f(p)$ in Conjecture D. For the moment, we just mention that we believe that the conjecture should hold with $f(p)=1/2$ if $p>3$ but that the best possible bound for $p=3$ is $f(3)=4/7$.
 
Next, we consider an asymptotic version of Theorem C.
 
 \begin{thmE}
 Let $p$ be a prime.
 If $\mathcal{S}=\{(G,H)\}$ is a sequence where $G$ is simple of order divisible by $p$ and $H<G$ contains a Sylow $p$-subgroup of $G$, then $\nu_p(H)/\nu_p(G)\to0$ when $(G,H)\in\mathcal{S}$ and $|G|\to\infty$. 
 \end{thmE}

 The above example $(G,H) = (\mathrm{SL}(2,2^k), D_{2(2^k-1)})$ shows that in Theorem E the hypothesis that $H$ contains a Sylow $p$-subgroup is necessary, when $p=2$. 
 
 Our second type of application of the fixed point ratio results concerns covering the set $G_p$ of $p$-elements of a finite group $G$. The problem of covering a finite group by proper subgroups has a long history that dates back to at least 1926 \cite{sco}. The problem of finding the smallest number of proper subgroups  of a finite group $G$ (generated by  $p$-elements) that are necessary to cover $G_p$ was introduced in \cite{mmm}. Let $\sigma_p(G)$ be this number. Using the $p$-solvable case of Theorem C, it was shown in \cite[Theorem B]{mmm} that for non-cyclic $p$-solvable groups we have $\sigma_p(G)\geq p+1$. Clearly, this bound is best possible, as any $2$-generated non-cyclic finite $p$-group shows. We extend this result to arbitrary groups.
 
 \begin{thmF}
 If $p$ is a prime and $G$ is a finite group generated by its $p$-elements, then $\sigma_p(G)\geq p+1$. 
 \end{thmF}
 
This result, which was conjectured in an earlier version of \cite{mmm},  is elementary. Finally, we prove an asymptotic result that relies on Theorem B.

\begin{thmG}
Let $p$ be a prime. Then $\sigma_p(G)\to\infty$ when $G$ is simple of order divisible by $p$ and $|G|\to \infty$.
\end{thmG}

In fact, we will prove that the number of proper subgroups that are necessary to cover some conjugacy class of $p$-elements goes to infinity.
We will also show that there are counterexamples when $G$ is not simple. 

In Section \ref{graph}, we will discuss some applications of Theorem G, and therefore of our fixed point ratio results, to noncommuting sets and to the existence of abelian Hall $\pi$-subgroups. This will also be related with the commuting probability of $p$-elements studied in \cite{bgmn} and with a theorem of Tur\'an \cite{T} on graph theory. For instance, we will prove the following local version of Theorem 1.1 of \cite{aahz06}. If $\pi$ is a set of primes, we will say that $G\in C(m,n)$ if for every $S_1,S_2\subseteq G_{\pi}$ with $|S_1|=m, |S_2|=n$, there exist $x\in S_1, y\in S_2$ such that $xy=yx$. 

\begin{thmH}
There exists a real-valued function $g(m,n)$ such that if $\pi$ is a set of primes, $G\in C_{\pi}(m,n)$ and $|G|_{\pi}>g(m,n)$ then $G$ possesses abelian Hall $\pi$-subgroups.
\end{thmH}

Note that by a theorem of Wielandt \cite{wie}, if a finite group has an abelian Hall $\pi$-subgroup, then all Hall $\pi$-subgroups are conjugate and every $\pi$-subgroup is contained in some Hall $\pi$-subgroup.

We will also prove that an affirmative answer to the question raised at the end of \cite{bgmn} on the commuting probability of $p$-elements implies a local version of a theorem of B. H. Neumann \cite{neu}, which 
 answered a question of Erd\H os (see Theorem \ref{localneu}).

The layout of the paper is as follows. We prove Theorem A in Section 3 and Theorem B in Section 4. We use these results to prove the theorems on Sylow numbers in Section 5 and those on covering the set of $p$-elements in Section 6. We conclude in Section 8 with examples showing that most of the hypotheses in our results are necessary.

\section{Fixed point ratios for primitive groups}

Let $G$ be a transitive permutation group acting on a finite set $\Omega$ with point stabilizer $H$. For $x\in G$, the fixed point ratio of $x$ on $\Omega$ is defined to be
$$\fpr(x,\Omega)=\frac{|\bC_{\Omega}(x)|}{|\Omega|}.$$ 

A powerful theorem of Burness and Guralnick \cite{bg} provides an explicit upper bound for $\fpr(x,\Omega)$ in case $x$ is an element of prime order of a primitive permutation group acting on $\Omega$. 

\begin{thm}[Theorem 1 of \cite{bg}]\label{fpr1}
	Let $G$ be a primitive permutation group acting on a finite set $\Omega$. Let $H$ be a point stabilizer in $G$. If $x \in G$ is an element of prime order $p$, then either
	$$\fpr(x,\Omega)\leq \frac{1}{p+1},$$
	or one of the following holds:
	
	\begin{itemize}
		
		\item [(i)] $G$ is almost simple and one of the following holds.
		
		\begin{itemize}
			\item [a)] $G\in \{S_n, A_n \}$ and $\Omega$ is the set of $t$-element subsets of a set of size $n$, with $1\leq t < n/2$.
			
			\item [b)]  $G$ is a classical group in a subspace action and    $(G,H,x,\fpr(x,\Omega))$ is listed in \cite[Table 6]{bg}.
			
			\item [c)]  $G=S_n$, $H=S_{n/2}\wr C_2$ and $x$ is a transposition.
			
			\item [d)]  $G= \mathrm{M}_{22}:2$, $H=\mathrm{L}_3(4).2_2$ and $x \in 2B$.
		\end{itemize}
		
		\item [(ii)] $G$ is an affine group with socle $(C_p)^d$, the element $x \in \GL(d,p)$ is a transvection and $\fpr(x, \Omega) = 1/p$.
		
		\item [(iii)] $G \leq L \wr S_k$ is a product type group with its product action on $\Omega=\Gamma^k$, $k \geq 2$ and $x \in L^k \cap G$, where $L \leq \Sym(\Gamma)$ is one of the almost simple primitive groups in part (i).
	\end{itemize}
\end{thm}

Note that this form of part (iii) of Theorem \ref{fpr1} follows from \cite[Remark 1(c)]{bg}.   

A purpose of this section is to prove the following result for primitive groups.

\begin{thm}
\label{oldA}	
	Let $G$ be a primitive permutation group acting on a finite set $\Omega$. Let $p$ be a prime divisor of the size of $G$. Then there exists an element $y$ in $G$ of order dividing $p^2$ such that $$\fpr(y,\Omega) \leq \frac{p-1}{2p-1}.$$ Moreover, at least one of the following holds.
	\begin{itemize}
		\item[(i)] There exists an element $x$ in $G$ of order $p$ such that
		$\fpr(x,\Omega) \leq 1/(p+1)$.   
		
		\item[(ii)] $p$ is odd and $G \leq S_{m} \wr S_{k}$ is a product type primitive group with its product action on $\Omega = \Gamma^{k}$, where $1 \leq k \leq \log_{2}(p+1)$, the set $\Gamma$ is the set of $t$-element subsets of a set of size $m$, and $p+1 < m < p^{2}-p$. In this case there exists an element $x$ of order $p$ in $G$ such that $\fpr(x,\Omega) \leq (p-1)/(2p-1)$. 
		
		\item[(iii)] The prime $p$ is a Mersenne prime and the group $G$ is almost simple with socle $\mathrm{SL}(2,p+1)$ and there exists an element $x$ of order $p$ in $G$ such that $\fpr(x,\Omega) = 2/(p+2)$. 
		
		\item[(iv)] $p=2$, the group $G$ is almost simple with socle $\mathrm{SL}(3,2)$, there exists an involution $x$ in $G$ with $\fpr(x,\Omega) = 3/7$ and there exists an element $y$ of order $4$ in $G$ with $\fpr(y,\Omega) = 1/7$.
		
		\item[(v)] $p = 2$, the group $G$ is $A_6$ acting on $6$ points or is $A_7$ acting on $7$ points, there exists an element $x$ of order $2$ in $G$ with $\fpr(x,\Omega) \leq 3/7$ and there exists an element $y$ of order $4$ in $G$ with $\fpr(y,\Omega) = 0$ when $G = A_{6}$ and with $\fpr(y,\Omega) = 1/7$ when $G = A_{7}$. 
		
	\end{itemize} 
\end{thm}

We need two elementary lemmas before we can start the proof of Theorem \ref{oldA}.

\begin{lem}
	\label{SumComb}
	If $c_1,\ldots ,c_t, d_1, \ldots ,d_t$ are non-negative integers with $c_i\geq d_i$ for all $i$ with $1 \leq i \leq t$, then 
	$$\prod_{i=1}^{t}{c_i \choose d_i}\leq {\sum_{i=1}^{t} c_i \choose \sum_{i=1}^{t}d_i}.$$
\end{lem}

\begin{lem}\label{Case}
	If $a,b>0$ and $r\geq 0$ are non-negative integers such that $r \leq a\leq b$, then 
	$$\frac{{a \choose r}}{{b \choose r}}\leq \frac{a}{b}.$$
\end{lem}

\begin{proof}
	We observe that  $\frac{a-i}{b-i}\leq \frac{a}{b}$ for any $i\geq 1$. It follows that
	$$\frac{{a \choose r}}{{b \choose r}}=\frac{a(a-1)\ldots (a-r+1)}{b(b-1)\ldots (b-r+1)}\leq \Big(\frac{a}{b}\Big)^r\leq \frac{a}{b}.$$
\end{proof}

\begin{proof}[Proof of Theorem \ref{oldA}]
The first statement of the theorem follows from the second. We will prove the second claim. 

We may assume by Theorem \ref{fpr1} that Case (i), (ii) or (iii) holds in Theorem \ref{fpr1}, otherwise Case (i) of the theorem holds. If any of the Cases (i)(c), (i)(d) or (ii) holds, then there exists an element $y$ in $G$ such that $\fpr(y,\Omega) \leq 1/(p+1)$ and again Case (i) holds.  

Assume that Case (i)(a) of Theorem \ref{fpr1} holds. Write $n$ in the form $p \ell + r$ and write $t$ in the form $p \ell_{0} + r_{0}$ for non-negative integers $\ell$, $\ell_{0}$, $r$, $r_{0}$, where $r$ and $r_{0}$ are at most $p-1$. For $p$ odd, let $x$ be an element of $G$ (of order $p$) whose disjoint cycle decomposition consists of $\ell$ cycles of length $p$ and $r$ cycles of length $1$. For $p=2$, let $x$ be an element of $G$ of order $2$ whose disjoint cycle decomposition consists of $\ell$ cycles of length $2$ if $\ell$ is even and $\ell - 1$ cycles of length $2$ when $\ell$ is odd and $G = A_{n}$. 

Let $p$ be odd or let $p=2$ and $\ell$ be even. In this case the fixed point ratio of $x$ is $$\frac{\binom{\ell}{\ell_{0}} \binom{r}{r_{0}}  }{\binom{n}{t}} \leq \frac{\binom{\ell}{\ell_{0}} \binom{r}{r_{0}}  }{{ \binom{\ell}{\ell_{0}}}^{p} \binom{r}{r_{0}} } = {\binom{\ell}{\ell_{0}}}^{1-p}$$ by Lemma \ref{SumComb}. This is at most $1/(p+1)$ unless possibly if $\binom{\ell}{\ell_{0}} = 1$ or $\binom{\ell}{\ell_{0}} = 2$ and $p=2$. Assume that an exceptional case holds. Let $p=2$. Clearly, $n \geq 5 = 2p+1$ and so $\binom{n}{t} \geq 2p+2$ or $(n,t) = (5,1)$. In any case $\fpr(x,\Omega) = \binom{\ell}{\ell_{0}}/\binom{n}{t} \leq 1/(p+1)$. Now let $p$ be odd. Since $\binom{\ell}{\ell_{0}} = 1$ and $2t \leq n$, we have $\ell_{0} = 0$ and 
\begin{equation}
\label{seged1}	
\fpr(x,\Omega) = \frac{\binom{r}{r_{0}}}{\binom{n}{t}} = \frac{\binom{r}{t}}{\binom{n}{t}} \leq \frac{r}{n} = \frac{r}{p\ell + r}
\end{equation}
by Lemma \ref{Case}. This is at most $1/(p+1)$ provided that $\ell \geq p-1$. This occurs when $n \geq p^{2}-p$. It is easy to see that $\fpr(x,\Omega) \leq 1/(p+1)$ provided that $n \in \{ p, p+1 \}$. For all other values of $n$, Case (ii) holds. 

Let $p=2$, let $\ell \geq 3$ be odd and let $G = A_{n}$. Since the $t$-element 
subsets (in $\Omega$) fixed by $x$ consist of $i$ $2$-cycles of $x$ and 
$j$ fixed points of $x$ such that $t = 2i+j$, we have 
\begin{equation}
\label{equation1}	
\fpr(x,\Omega) = \frac{\sum_{t=2i+j} \binom{\ell - 1}{i} \binom{r+2}{j}}{\binom{n}{t}}.
\end{equation}
For $t=1$ the fraction in (\ref{equation1}) is at most $1/3$, unless $n \in \{ 6, 7 \}$, when Case (v) holds. If $t=2$, then (\ref{equation1}) gives
$$\fpr(x,\Omega) \leq \frac{\ell + 2 }{\binom{n}{2}} = \frac{2\ell + 4}{n(n-1)} \leq \frac{n+4}{n(n-1)} \leq \frac{1}{3}.$$ A similar calculation shows that the bound $\fpr(x,\Omega) \leq 1/3$ also holds for $t=3$. 

Let $t \geq 4$. Since $t < n/2$, the integer $n$ must be at least $9$. Since $\ell$ is odd, it must then be at least $5$ (and so $n \geq 10$). Fix non-negative integers $i$ and $j$ such that $t = 2i + j$, $i \leq \ell -1$ and $j \leq r+2$. We have 
$$\binom{n}{t} = \binom{2 (\ell - 1) + (r+2)}{2i+j} \geq \binom{\ell -1}{i}^{2} \binom{r+2}{j}$$ by Lemma \ref{SumComb}. This and (\ref{equation1}) give $$\fpr(x,\Omega) \leq \sum_{t=2i+j} \binom{\ell - 1}{i}^{-1} \leq \frac{2}{\ell - 1} \leq \frac{1}{3}$$ provided that $\ell \geq 7$. Let $\ell = 5$ and so $n = 10$ or $n = 11$ and $t=4$. In this case $$\fpr(x,\Omega) = \frac{\sum_{4=2i+j} \binom{4}{i} \binom{r+2}{j}}{\binom{n}{4}} = \frac{6 + 2 (r+2)(r+1)}{\binom{n}{4}} \leq \frac{18}{\binom{10}{4}} < \frac{1}{3}$$ by (\ref{equation1}).

Let Case (i)(b) of Theorem \ref{fpr1} hold. Consider the various lines of \cite[Table 6]{bg}. With the exceptions of groups $G$ in lines 1, 2 and 3, an element $x$ of order $p$ may be chosen from $G$ such that $\fpr(x,\Omega) \leq 1/(p+1)$ holds. Case (i) holds. In the remaining of this paragraph, we will assume that Case (i) does not hold. If $G$ is of line 1 in \cite[Table 6]{bg}, then Case (iv) holds. If $G$ is of line 2 in \cite[Table 6]{bg}, then Case (iii) holds. If $G$ is of line 3 in \cite[Table 6]{bg}, then the exceptional element of order $3$ is a field automorphism of $\mathrm{SL}(2,8)$ but the socle $\mathrm{SL}(2,8)$ contains an element of order $3$ and so Case (i) holds.  

Finally, let Case (iii) of Theorem \ref{fpr1} hold. The group $L$ is of Case (i)(a), (i)(b), (i)(c) of Theorem \ref{fpr1}. If $L$ is of Case (i)(b) or Case (i)(c) of Theorem \ref{fpr1}, then Case (i) holds by \cite[Proposition 6.2]{bg}. Let $L$ be of Case (i)(a) of Theorem \ref{fpr1}. For $p=2$ we proved above that there exists an element $y$ of order $2$ in $L$ such that $\fpr(y,\Gamma) \leq 1/3$ and so Case (i) holds for $G$ by \cite[Remark 1(c)]{bg}. Let $p$ be odd. By the above computation and \cite[Remark 1(c)]{bg} (and by replacing $n$ with $m$), we have Case (i) or $G \leq S_{m} \wr S_{k}$ is a product type primitive group with its product action on $\Omega = \Gamma^{k}$, where $1 \leq k$, the set $\Gamma$ is the set of $t$-element subsets of a set of size $m$, and $p+1 < m < p^{2}-p$. Assume that the latter case holds. Let $k=1$. If Case (i) does not hold, then there exists an element $x$ of order $p$ in $G$ such that $$\fpr(x,\Omega) \leq \frac{r}{m} \leq \frac{p-1}{2p-1} < \frac{1}{2}$$ by (\ref{seged1}), where $m = \ell p + r$ for non-negative integers $\ell$ and $r$ with $0 \leq r \leq p-1$. In general, for any $k$, there exists an element $x$ of order $p$ in $G$ such that $\fpr(x,\Omega) < 2^{-k}$. If $k \leq \log_{2}(p+1)$, then Case (ii) holds, otherwise Case (i) holds.      
\end{proof}

We will need the following consequence of Theorem \ref{oldA}.

\begin{cor}
\label{corprimitive}	
Let $G$ be a primitive permutation group acting on a finite set $\Omega$. Let $p$ be a prime divisor of the size of $G$.
\begin{itemize}
	\item[(i)] If $p=2$, then there exists an element $x$ of order dividing $4$ in $G$ such that $\fpr(x,\Omega) \leq 1/3$. 
	\item[(ii)] If $p > 2$, the group $G$ is generated by its $p$-elements and $G$ is not isomorphic to $A_m$ with $p+1 < m < p^{2}-p$ and not isomorphic to $\mathrm{SL}(2,p+1)$ where $p$ is a Mersenne prime, then there exists an element $x$ of order $p$ in $G$ such that $\fpr(x,\Omega) \leq 1/(p+1)$. 
\end{itemize}	
\end{cor}

\begin{proof}
Statement (i) is a direct consequence of Theorem \ref{oldA}. Consider statement (ii). Let $p > 2$. We may assume that the group $G$ is of Case (ii) or (iii) of Theorem \ref{oldA}. Since $G$ is generated by its $p$-elements, $G = A_{m}$ with $p+1 < m < p^{2}-p$ in Case (ii) and $G = \mathrm{SL}(2,p+1)$ for a Mersenne prime $p$ in Case (iii). However, these cannot occur by our hypothesis. 
\end{proof}

\section{Proof of Theorem A}

In this section we will prove Theorem A. 

Let $G$ be a transitive permutation group on a finite set $\Omega$. Let $p$ be a prime divisor of the size of $G$. Let $P$ be a Sylow $p$-subgroup of $G$.   

Let $\mathcal{B}$ be a (maximal) system of blocks of imprimitivity for $G$ such that the action of $G$ on $\mathcal{B}$ is primitive. For any element $x$ in $G$ we have $\fpr(x,\Omega) \leq \fpr(x,\mathcal{B})$ (because of $\omega\in\Omega$ is fixed by $x$ and $B$ is the block containing $\omega$, then $x$ fixes $B$). Let the kernel of the action of $G$ on $\mathcal{B}$ be $K$. If $G$ is generated by its $p$-elements, then the factor group $G/K$ is also generated by its $p$-elements. 

Part (i) of Theorem A follows from the previous paragraph and Theorem \ref{oldA}. Part (ii) of Theorem A follows from the previous paragraph and Corollary \ref{corprimitive}. 

Let $p>2$. Assume first that the transitive group $G$ is generated by its $p$-elements. In this special case, in order to prove (iii) and (iv) of Theorem A, we may assume that $G$ is a primitive permutation group. There exists an element $x$ of order $p$ in $G$ by Theorem \ref{oldA} such that $\fpr(x,\Omega) \leq 1/(p+1)$, except if $G = A_{m}$ with $p+1 < m < p^{2}-p$, when $\fpr(x,\Omega) \leq (p-1)/(2p-1)$. Let the number of orbits of the cyclic group $H$ generated by $x$ acting on $\Omega$ be $\ell$. We have
$$\ell = \frac{|\Omega|}{|H|} \sum_{h \in H} \fpr(h,\Omega) \leq \frac{|\Omega|}{|H|} ( 1 + (|H|-1) \fpr(x,\Omega)  ) =  \frac{|\Omega|}{p} ( 1 + (p-1) \fpr(x,\Omega)  ).$$  
Since $x$ may be taken to be in $P$, we see that the number of orbits of $P$ on $\Omega$ is at most $(p/(2p-1)) |\Omega|$ and is at most $(2/(p+1)) |\Omega|$ when $G$ is different from $A_{m}$ with $p+1 < m < p^{2}-p$. 

We turn to the proof of (iii) and (iv) of Theorem A. Let $L$ be the subgroup of the transitive group $G$ which is generated by all $p$-elements of $G$. Let the number of orbits of $L$ on $\Omega$ be $s$. Let these orbits be $\Omega_{1}, \ldots , \Omega_{s}$. The subgroup $P$ is contained in $L$. For each $i$, the number of orbits of $P$ on $\Omega_{i}$ is at most $(p/(2p-1)) |\Omega_{i}|$ by the previous paragraph. It follows that the number of orbits of $P$ on $\Omega$ is at most $(p/(2p-1)) |\Omega|$. This proves (iii) of Theorem A.

Assume that $L$ has no factor group isomorphic to an alternating group $A_m$ with $p+1 < m < p^{2}-p$. Since $P \leq L$, we know by the above that $P$ has at most $(2/(p+1))|\Omega_{i}|$ orbits on $\Omega_{i}$ for every $i$ with $1 \leq i \leq s$. It follows that $P$ has at most $(2/(p+1))|\Omega|$ orbits on $\Omega$. This establishes (iv)(a).  

Finally, we turn to the proof of (iv)(b). The Orbit-Counting Lemma gives
\begin{equation}
\label{seged2}	
\frac{2}{p+1} \geq \frac{1}{|P|} \sum_{g \in P} \fpr(g,\Omega)
\end{equation}
by (iv)(a). Let $\mu$ be $\min_{g \in P} \fpr(g,\Omega)$. This is less than $1$. We obtain  
$$\frac{2}{p+1} \geq \frac{1}{|P|} (1 + \mu (|P|-1)) = \mu + \frac{1-\mu}{|P|} > \mu$$ by (\ref{seged2}). This proves (iv)(b).

The proof is complete. 

\section{Proof of Theorem B}

In this section we will prove Theorem B. 

We may assume that $G$ is non-abelian. We may also assume that $G$ is a primitive permutation group by the first observation in the previous section. 

We start with the case when $G$ is an alternating group. 

In this paragraph we will define an element $x$ in $A_n$. Let the $p$-adic decomposition of $n$ be $a_{f}p^{f} + \ldots + a_{1}p + a_{0}$ where $f$ is an integer, $0 \leq a_{i} \leq p-1$ for all $i$ with $0 \leq i \leq p-1$ and $a_{f} > 0$. Let $p$ be odd. We define $x \in S_{n}$ to be a permutation whose disjoint cycle decomposition possesses $a_{i}$ cycles of length $p^{i}$ for every $i$ with $0 \leq i \leq f$. This is a $p$-element belonging to $A_{n}$. Let $p=2$. There are two cases. If $\sum_{i=1}^{f} a_{i}$ is even, then let $x$ be a permutation in $S_{n}$ whose disjoint cycle decomposition has $a_{i}$ cycles of length $p^{i}$ for every $i$ with $0 \leq i \leq f$. Observe that $x \in A_{n}$. If $\sum_{i=1}^{f} a_{i}$ is odd, then choose $x$ to be a permutation whose disjoint cycle decomposition consists of $2 + a_{f-1}$ cycles of length $2^{f-1}$ and $a_{i}$ cycles of length $2^{i}$ for all $0 \leq i \leq f-2$. Observe again that $x \in A_{n}$. 

Let $k$ be a positive integer less than $n/2$. We also need the $p$-adic decomposition of the positive integer $k$ less than $n/2$. Write $k = \sum_{i=0}^{f} b_{i} p^{i}$ where $b_{i}$ is an integer with $0 \leq b_{i} \leq p-1$ for every $i$ with $0 \leq i \leq f$. 

Let $\Omega$ be the set of $k$-element subsets of a set of size $n$. Let $S$ be an element of $\Omega$. Observe that $S$ is fixed by $x$ if and only if $S$ is the union of some of the cycles of $x$. It follows that the number of elements in $\Omega$ fixed by $x$ is 
\begin{equation}
	\label{e1}	 
	\binom{a_{f}}{b_{f}} \binom{a_{f-1}}{b_{f-1}} \cdots \binom{a_{0}}{b_{0}},
\end{equation}
unless $p=2$ and $\sum_{i=1}^{f} a_{i}$ is odd. Note that in this formula it is implicitly assumed that $b_{i} \leq a_{i}$ for every $i$ with $0 \leq i \leq f$, otherwise $x$ acts fixed-point freely on $\Omega$. We have 
\begin{equation}
	\label{e0}	
	\fpr(x,\Omega)=\frac{{a_f\choose b_f}{a_{f-1}\choose b_{f-1}}\ldots {a_{0}\choose b_{0}}}{{n\choose k}}
\end{equation}
by (\ref{e1}). 

\begin{prop}
\label{palternating}	
Let $G = A_{n}$ with $n \geq 5$. Let $p$ be a prime divisor of the size of $G$. There exists a $p$-element $x \in G$ such that $\fpr(x,\Omega) \to 0$ as $|G| \to \infty$ whenever $\Omega$ is a finite set on which $G$ acts primitively. 
\end{prop}

\begin{proof}
Let $x$ be the element in $G = A_{n}$ defined before the proposition. Note that the definition of $x$ depends on whether $p$ is $2$ or different from $2$, but the argument below works for any prime $p$. Let the $p$-adic decomposition of $n$ be $a_{f}p^{f} + \ldots + a_{1}p + a_{0}$ where $f$ is an integer, $0 \leq a_{i} \leq p-1$ for all $i$ with $0 \leq i \leq p-1$ and $a_{f} > 0$.  Since $p^{f} \leq n$, we have $f \leq \log_{p}n$. The number of cycles in the disjoint cycle decomposition of $x$ is at most $p (f + 1) \leq p (1 + \log_{p}n)$. It follows that 
\begin{equation}
\label{e7}
|\bC_{G}(x)| < n^{p (1 + \log_{p}n)} \cdot {((p+1)!)}^{f+1} < {((p+1)n)}^{p(1 + \log_{p}n)} < n^{c(p) \log_{2} n}
\end{equation}
for some constant $c(p)$ depending on $p$. 

Let $H$ be the stabilizer in $G$ of a point in $\Omega$. We have 
\begin{equation}
\label{e8}	
\fpr(x, \Omega) = \frac{|x^{G} \cap H|}{|x^{G}|} \leq \frac{|H| |\bC_{G}(x)|}{|G|} < \frac{|H|}{|G|} \cdot n^{c(p) \log_{2} n}
\end{equation}
by \cite[Lemma 1.2(iii)]{bur} and (\ref{e7}). If $H$ is primitive, then $$\frac{|H|}{|G|} \leq 2 \cdot {\Big( \Big[\frac{1}{2}(n+1)\Big]! \Big)}^{-1}$$ by a theorem of Bochert \cite{Bo}, and so $\fpr(x, \Omega) \to 0$ by (\ref{e8}). If $H$ is imprimitive and $n \not= 6$, then $$\frac{|H|}{|G|} \leq 2^{1 -[(n+1)/2]}$$ by the claim in \cite[Lemma 2.1]{Ma}, and so, once again, $\fpr(x, \Omega) \to 0$ by (\ref{e8}). Finally, let $H$ be intransitive. The group $H$ is conjugate to $(S_{n-k} \times S_{k}) \cap G$ for some integer $k$ with $1 \leq k < n/2$. If $p=2$, then by the comments above, we have that $\fpr(x,\Omega)\leq 3/n$, which tends to $0$ as $n$  goes to infinity. Thus, we may assume that $p$ is odd. Observe that $b_{j} = 0$ in (\ref{e0}) provided that $j \geq \log_{2}k \geq \log_{p}k$. On the other hand, since each $a_j$ in (\ref{e0}) is less than $p$, each binomial coefficient in the numerator of the right-hand side of (\ref{e0}) is less than $2^p$. These imply that 
$$\fpr(x,\Omega) < \frac{2^{p \log_{2}k}}{\binom{n}{k} } = \frac{k^{p}}{\binom{n}{k}}.$$ Now $k^{p}/\binom{n}{k} \leq k^{p}/n$. This tends to $0$ provided that $k$ is bounded. We may also write $k^{p}/\binom{n}{k} \leq k^{p}/{(n/k)}^{k} \leq k^{p}/2^{k}$. This tends to $0$ if $k$ goes to infinity.
\end{proof}

The following is an important result of Liebeck and Saxl \cite{LSax} on fixed point ratios of actions of groups of Lie type. 

\begin{thm}[Theorem 1 of \cite{LSax}]\label{Lsax}
Let $G$ be an almost simple group with socle a simple group of Lie type defined over the field of size $q$ (or the field of size $q^{2}$ when $G$ is unitary). If $G$ acts primitively on a finite set $\Omega$, then either 
$$\fpr(x,\Omega)\leq \frac{4}{3q}$$ for every non-trivial element $x$ in $G$ or $(G,\Omega,x)$ is listed in Table 1 of \cite{LSax}.
\end{thm}

By inspecting Table 1 of \cite{LSax}, it is easy to deduce the following result.

\begin{cor}\label{LsaxV2}
There exists a universal constant $C$ such that whenever $G$ is an almost simple group with socle a simple group of Lie type defined over the field of size $q$ (or the field of size $q^{2}$ when $G$ is unitary) and $G$ acts primitively on a finite set $\Omega$ then $$\fpr(x,\Omega)\leq \frac{C}{\sqrt{q}}$$ for every non-trivial element $x$ in $G$.
\end{cor}

Corollary \ref{LsaxV2} immediately implies the following. 

\begin{prop}
\label{pexceptional}	
Let $G$ be a finite simple group of Lie type defined over the field of size $q$ (or the field of size $q^{2}$ when $G$ is unitary). For every non-trivial element $x \in G$ and for every set $\Omega$ on which $G$ acts primitively, we have $\fpr(x,\Omega) \to 0$ as $q \to \infty$. 
\end{prop}

Another fundamental result on this topic is the following theorem of Liebeck and Shalev \cite{LSh}.  

\begin{thm}[Theorem $(\star)$ of \cite{LSh}]\label{LSh}
There exists a constant $\epsilon>0$ such that whenever $G$ is an almost simple
classical group acting primitively on a finite set $\Omega$ and the action is not a subspace action, then 

$$\fpr(x,\Omega)\leq |x^G|^{-\epsilon}$$ for every non-trivial element $x$ in $G$. 
\end{thm}

In a series of papers Burness  \cite{b1} proved that we can take $\epsilon \approx 1/2$.

Since $|x^{G}| \to \infty$ as $|G| \to \infty$, in the statement of Theorem \ref{LSh}, Theorem \ref{LSh} immediately implies the following.

\begin{prop}
\label{pnonsubspace}	
Let $G$ be a finite simple classical group. For every non-trivial element $x \in G$ and for every set $\Omega$ on which $G$ acts primitively with a non-subspace action, we have $\fpr(x,\Omega) \to 0$ as $|G| \to \infty$. 
\end{prop}

We will need the following theorem of Frohardt and Magaard \cite[Main theorem]{FM}.

\begin{thm}[Main theorem of \cite{FM}]
\label{tsubspace}	
Let $G = \mathrm{Cl}(n,q)$ be a finite simple classical group with $q$ bounded. Let $x \in G$ be an element whose lift is a linear transformation of the underlying natural module, with an eigenspace of codimension $d$. Then $\fpr(x,\Omega) \to 0$ provided that $d$ (and $n$) tends to infinity whenever $\Omega$ is a set on which $G$ acts primitively and the action is a subspace action. 
\end{thm}
	
This has the following consequence. 

\begin{prop}
	\label{psubspace}	
	Let $G$ be a finite simple classical group defined over a field of bounded size. Let $p$ be a prime dividing the order of $G$. There exists a $p$-element $x \in G$ such that whenever $\Omega$ is a set on which $G$ acts primitively with a subspace action, we have $\fpr(x,\Omega) \to 0$ as $|G| \to \infty$. 
\end{prop}	
	
\begin{proof}	
The finite classical simple group $G = \mathrm{Cl}(n,q)$ has lift $M$ with natural module $V$ of dimension $n$ defined over the field of size $q$ (or $q^{2}$ in the unitary case). The groups $G$ and $M$ are both linear, symplectic, orthogonal or unitary groups and the space $V$ is a linear space or a nondegenerate symplectic, orthogonal or unitary space. Let $p$ be a prime divisor of the order of $G$. If $V$ is a linear, a symplectic, a unitary or an orthogonal space not of minus type, then let $r$ be the smallest integer such that the nonabelian simple group $G_{r} = \mathrm{Cl}(r,q)$ has order divisible by $p$. Let $M_{r}$ be the lift of $G_r$ with natural module $V_r$. Let $x_r$ be a noncentral $p$-element in $G_r$ with a lift $m_r$ in $M_r$. This element has an eigenspace of codimension $d_{r} \geq 1$. Write $n$ in the form $kr + \ell$ where $k$ and $\ell$ are integers with $0 \leq \ell \leq r-1$. The module $V$ may be written as a direct sum of $k$ copies of $V_r$ with the sum $I_{\ell}$ of $\ell$ copies of the trivial module. Let $m$ be the element of $M$ which acts as $m_r$ on each copy of $V_r$ and acts trivially on $I_{\ell}$. The element $m$ has an eigenspace of codimension at least $k d_{r}$. Since $n$ goes to infinity, so does $k$, and the image $x$ of $m$ in $G$ has an eigenspace of codimension approaching infinity. This implies that $\fpr(x,\Omega) \to 0$ by Theorem \ref{tsubspace}, as $n \to \infty$, where $\Omega$ is any set on which $G$ acts in a subspace action. Finally, let $V$ be an orthogonal space of minus type. In this case the only change in the preceding argument is that $n$ is written in the form $kr + \ell$ where $k$ is odd and $0 \leq \ell \leq 2r-1$.  
\end{proof}

\begin{proof}[Proof of Theorem B]
Let $\{ G_m \}_{m =1}^{\infty}$ be a sequence of finite simple groups with $|G_{m}| \to \infty$. Let $p$ be a prime dividing $|G_{m}|$ for every positive integer $m$. We may assume that $G_{m}$ is non-abelian for every $m$. For every fixed constant $C$, we may remove from the sequence $\{ G_m \}_{m =1}^{\infty}$ every group of size bounded by $C$. We may remove from the sequence $\{ G_m \}_{m =1}^{\infty}$ every alternating group by Proposition \ref{palternating}. Every member $G_m$ of $\{ G_m \}_{m =1}^{\infty}$ is a simple group of Lie type defined over the field of size $q_m$ (or the field of size $q_{m}^{2}$ when $G_m$ is unitary). In fact, we may assume by Proposition \ref{pexceptional} that each $G_m$ is a simple classical group $\mathrm{Cl}(n_{m},q_{m})$ whose lift has a natural module of dimension $n_{m}$ over the field of size $q_{m}$ (or the field of size $q_{m}^{2}$ when $G_{m}$ is unitary). Moreover, we may assume that each $q_m$ is bounded by a universal constant by Proposition \ref{pexceptional}. We may also assume by Proposition \ref{pnonsubspace} that each $G_m$ acts primitively on a set $\Omega_m$ such that the action is a subspace action. Theorem B now follows from Proposition \ref{psubspace}.
\end{proof}

\section{Sylow numbers}

In this section we prove Theorems C and E and describe the structure of a minimal counterexample to Conjecture D. For a finite group $G$ and a prime $p$, let $\nu_{p}(G)$ be the number of Sylow $p$-subgroups of $G$. 

We begin with a lemma which was mentioned in the Introduction.

\begin{lem}
	\label{lessorequal}	
	Let $p$ be a prime. If $H$ is a subgroup of a finite group $G$, then $$\nu_{p}(H) \leq \nu_{p}(G).$$ Furthermore, 
	$\nu_p(H)=\nu_p(G)$ if and only if the following hold:
\begin{enumerate}
\item
Any Sylow $p$-subgroup of $H$ is contained in a unique Sylow $p$-subgroup of $G$.
\item
$G=H\bN_G(P)$.
\end{enumerate}
\end{lem}

\begin{proof}
	Let $Q$ be a Sylow $p$-subgroup of $H$ and let $P$ be a Sylow $p$-subgroup of $G$ containing $Q$. Since $\norm G P \cap H \leq \norm H Q$, we have $$|G| \geq |H \norm G P| = \frac{|H||\norm G P|}{|H \cap \norm G P|} \geq \frac{|H||\norm G P|}{|\norm H Q|},$$ and so $\nu_{p}(G) = |G : \norm G P| \geq |H : \norm H Q| = \nu_{p}(H)$. 	In particular, this shows that $\nu_p(H)=\nu_p(G)$ if and only if 
$G=H\bN_G(P)$)  and $\bN_H(P) = \bN_H(Q)$.

Suppose that $\nu_p(H)=\nu_p(G)$. We want to show (i).  
If $P_1$ and $P_2$ are two distinct 
Sylow subgroups of $G$ both containing $Q$, then $H\cap P_2 = Q = H \cap P_1$. 
There exists $h \in H$  such that $(P_1)^h=P_2$, but normalizes $Q$. This 
contradicts $\bN_H(P_1) = \bN_H(Q)$.

Conversely, (i) implies $\bN_H(P)=\bN_H(Q)$. This and (ii) gives $\nu_p(H)=\nu_p(G)$.
\end{proof}

Note that if $G=\SL(2,2^k)$, with $k\geq 2$ and $H$ is the normalizer of a nonsplit torus then $\nu_p(H)=\nu_p(G)$ for $p=2$.  For $p=3$, we have that $\nu_3(A_5)=\nu_3(A_6)$. Using \cite{gkl}, it is possible to deduce that there is no almost simple example when $p>3$. It would be interesting to classify the groups $G$ generated by $p$-elements  with a proper subgroup $H$ such that $\nu_p(H)=\nu_p(G)$.

We next prove Theorem C.  In the special case of $p$-solvable groups, we have a stronger statement.  

\begin{thm}\label{PsolvSylows}
Let $G$ be a $p$-solvable group for a prime $p$ and let $H$ be a proper subgroup of $G$. Then $\nu_p(H)$ divides $\nu_p(G)$. Furthermore, if $\nu_p(H) \not= \nu_p(G)$, then $$\nu_p(H)\leq \frac{\nu_p(G)}{p+1}.$$ In particular, Theorem C is true in case $G$ is a $p$-solvable group.
\end{thm}

\begin{proof}
The fact that  $\nu_p(H)$ divides $\nu_p(G)$ is Theorem A of Navarro \cite{nav}. It follows by Sylow's theorem that if $\nu_p(H)k = \nu_p(G)$, then $k$ is congruent to $1$ modulo $p$. The second statement follows. For the proof of Theorem C, we may assume that $\nu_p(H) = \nu_p(G)$, that $H$ contains a Sylow $p$-subgroup of $G$, and that $G$ is generated by $p$-elements. This cannot happen for $H$ a proper subgroup of $G$.
\end{proof}

Let $p$ be a prime and let $G$ be a finite group whose order is divisible by $p$. Let $H$ be a proper subgroup of $G$ containing a Sylow $p$-subgroup $P$ of $G$. In order to prove Theorem C, we may assume that $H$ is a maximal subgroup in $G$ by Lemma \ref{lessorequal}.  

We write $\Omega$ to denote the set of conjugates of $H$ in $G$. The group $G$ has a primitive action on the set $\Omega$. We define the fixed point ratio of $P$ on $\Omega$ to be 
$$
\fpr(P,\Omega)=\frac{|\bC_{\Omega}(P)|}{|\Omega|}.
$$
Note that for any element $x\in P$, we have $\fpr(P,\Omega)\leq \fpr(x,\Omega)$.

\begin{lem}\label{SylowFPR}
Let $p$ be a prime. Let $H$ be a maximal subgroup of $G$ containing a Sylow $p$-subgroup $P$ of $G$. Let $\Omega$ be the set of conjugates of $H$ in $G$. We have
$$\frac{\nu_p(H)}{\nu_p(G)}=\fpr(P,\Omega).$$
\end{lem}

\begin{proof}
Sylow's theorem gives
$$\frac{\nu_p(H)}{\nu_p(G)}=\frac{|H|}{|G|}\frac{|\norm G P|}{|\norm H P|}.$$
Observe that $|\Omega|=|G:H|$. Observe also that $\fpr(P,\Omega)$ is the number of conjugates of $H$ containing $P$. It is sufficient to show that this number is $|\norm G P|/|\norm H P|$.

Consider the set $\Sigma = \{ (Q,L) \ | \ Q \in \mathrm{Syl}_{p}(L) , \ L = H^{x} , \ x \in G \}$. There are $|G|/|H|$ conjugates of $H$ in $G$ and each conjugate contains $|H|/|\norm H P|$ Sylow $p$-subgroups. Thus $|\Sigma|=|G|/|\norm H P|$. On the other hand, $|\Sigma| = \fpr(P,\Omega) \cdot (|G|/|\norm G P|)$. The result follows.
\end{proof}

Theorem C follows by the assumption that $H$ is maximal in $G$, by Lemma \ref{SylowFPR}, by the line before Lemma \ref{SylowFPR} and by Theorem A. 

Next we complete the proof of Theorem E. 

\begin{proof}[Proof of Theorem E]
Let $G$ be a finite simple group of order divisible by $p$ and let $H$ be a subgroup of $G$ with $|H|_p=|G|_p$. If $H\leq L<G$, then $\nu_p(H)\leq \nu_p(L)$ by Lemma \ref{lessorequal}. We may thus assume that $H$ is a maximal subgroup of $G$ for every $(G,H)\in\mathcal{S}$. Let $P$ be a Sylow $p$-subgroup of $H$. For any element $x \in P$, we have $$\frac{\nu_p(H)}{\nu_p(G)}=\fpr(P,\Omega) \leq \fpr(x,\Omega)$$ by Lemma \ref{SylowFPR}, where $\Omega$ is the set of conjugates of $H$ in $G$. Sylow's theorem and Theorem B allows us to choose $x$ such that $\fpr(x,\Omega) \to 0$, as $|G| \to \infty$. 
\end{proof}

Finally, we describe the structure of a minimal counterexample to Conjecture D. We will need the following, which was proved in \cite{hal}, and we state here for the reader's convenience.

\begin{lem}\label{SylowNorm}
Let $p$ be a prime, let $G$ be a finite group, let $P\in\syl p G$ and let $N$ be a normal subgroup of $G$. Then 
$$\nu_p(G)=\nu_p(G/N)\nu_p(PN).$$
\end{lem}
\begin{proof}
See Equation (2.4) of \cite{hal}.
\end{proof}

We give the structure of a minimal counterexample to Conjecture D.

\begin{thm}\label{thm:minimalcounterConjD} 
Let $f(p)$ be a fixed function from the set of prime integers $p$ into $[1/2,1)$. 
Suppose that $(G,H)$ is a counterexample to Conjecture D, with this function $f(p)$,  with first $|G|$ and then $|G:H|$ as small as possible. Then $G$ has a unique minimal normal subgroup $A$ which is non-abelian and which has order divisible by $p$ such that $G=AQ$, where $Q\in\syl p H$. Furthermore, $H$ is maximal in $G$ and core-free.
\end{thm}

\begin{proof}
Let $G$ be a finite group  and let $H<G$ be a subgroup  of $G$ with $\nu_p(H)<\nu_p(G)$. Since $G$ is a minimal counterexample, we have that $\nu_p(J)\leq f(p)\nu_p(K)$ or $\nu_p(J)=\nu_p(K)$ whenever $J\leq K$ and $|K|<|G|$. Let $Q\in \Syl_p(H)$ and let $P\in \Syl_p(G)$ such that $Q=P\cap H$.

\medskip

\textit{Step 1: We may assume that $|G|>|H \norm G P|$ and $\norm H Q=\norm H P$.}

\medskip

Notice that
$$1>\frac{\nu_p(H)}{\nu_p(G)}=\frac{|H:\norm H Q|}{|G:\norm G P|}=\frac{|H\norm G P|}{|G|}\frac{|\norm H P|}{|\norm H Q|}.$$ Since $\norm H P$ is a subgroup of $\norm H Q$, if $|G|=|H \norm G P|$ or $|\norm H Q|>|\norm H P|$, we obtain that $|\norm H P|\leq |\norm H Q|/2$ and therefore, $$\nu_p(H)\leq \frac{1}{2}\nu_p(G),$$ which is not possible since $(G,H)$ is a counterexample. Therefore, we may assume that $|G|>|H \norm G P|$ and $\norm H Q=\norm H P$, as wanted. 
 
 \medskip
 
 \textit{Step 2: We may assume that $H$ is maximal.}
 
\medskip
 
Suppose that $H$ is not maximal and let $M$ be a maximal subgroup of $G$ containing $H$. If $\nu_p(H)<\nu_p(M)$, then $\nu_p(H)\leq f(p)\nu_p(M)$ by the minimality of $G$ as a counterexample. Thus, $\nu_p(H)\leq f(p)\nu_p(M)\leq f(p)\nu_p(G)$, by Lemma \ref{lessorequal}, which is a contradiction. Therefore, $\nu_p(H)=\nu_p(M)$, which contradicts the minimality of $|G:H|$.

\medskip

\textit{Step 3: $P$ is the unique Sylow $p$-subgroup in $G$ containing $Q$ and $\norm G Q\leq \norm G P$.}

\medskip

Assume that there exists $T\in \Syl_p(G)$ with $T\not=P$ such that $Q=P\cap H=T\cap H$. Let $S\in \Syl_p(H)$ with $S=Q^h$ for some $h \in H$.  Then $S=Q^h=P^h\cap H=T^h\cap H$ and $P^h,T^h\in \Syl_p(G)$. If $S\neq Q$, then $\{P,T\}\cap\{P^h,T^h\}=\emptyset$ and we conclude that for each Sylow $p$-subgroup of $H$, we can get $2$  Sylow $p$-subgroups of $G$ containing it and pairwise different. It follows that $\nu_p(H)\leq\frac{1}{2}\nu_p(G)$, a contradiction. Thus, we may assume that $P$ is the unique Sylow $p$-subgroup containing $Q$. Now, let $x\in\norm G Q$, then $Q=Q^x\subseteq P^x$, and it follows that $P=P^x$, so $\norm G Q\leq \norm G P$, as wanted.

\medskip

\textit{Step 4: $H$ is core-free.}

\medskip

Suppose that there exists $1< N$ a normal subgroup of $G$ contained in $H$. Then
$$\frac{\nu_p(H)}{\nu_p(G)}=\frac{\nu_p(H/N)\nu_p(N)\nu_p(\norm{QN}{Q\cap N})}{\nu_p(G/N)\nu_p(N)\nu_p(\norm  {PN}{P\cap N})},$$
by Theorem 2.1 of \cite{hal}. Notice that $P\cap N=Q\cap N$ and hence $\norm {QN}{Q\cap N}\leq \norm {PN}{P\cap N}$. Thus, $\nu_p(\norm {QN}{Q\cap N})\leq \nu_p(\norm {PN}{P\cap N})$ by Lemma \ref{lessorequal}. We claim that $\nu_p(H/N)<\nu_p(G/N)$. Since $\norm G Q\leq\norm G P$ by Step 3 and $P\cap H=Q$, we have that $\norm H P=\norm H Q$. Then $|N\norm G P:\norm G P|=|N\norm H Q:\norm H Q|$. Suppose that $\nu_p(H/N)=\nu_p(G/N)$, then $|G:N\norm G P |=|H:N\norm H Q |$ and therefore $\nu_p(G)=\nu_p(H)$, a contradiction.  Hence the claim is proved and since $|G/N|<|G|$, we obtain that $\nu_p(H/N)\leq f(p){\nu_p(G/N)}$ by the minimality of $G$ as a counterexample. Therefore, $\nu_p(H)\leq f(p) \nu_p(G)$, a contradiction.

\medskip

\textit{Step 5: We may assume that $G$ has a unique minimal normal subgroup $A$ and that $G=AQ$.}

\medskip

 Let $A$ be a minimal normal subgroup in $G$. Since $H$ is core-free and maximal we have that $G=HA$. Then
$$1>\frac{\nu_p(H)}{\nu_p(G)}=\frac{\nu_p(HA/A)\nu_p(Q(A\cap H))}{\nu_p(G/A)\nu_p(PA)}=\frac{\nu_p(Q(A\cap H))}{\nu_p(PA)},$$
by Lemma \ref{SylowNorm}. Notice that $Q(A\cap H)\not=PA$ since $Q(A\cap H)\leq H$ and $A\not \leq H$. Thus, $Q(A\cap H)<PA$. If $PA<G$, then $\nu_p(Q(A\cap H))\leq f(p) \nu_p(PA)$ by the minimality of $G$. Thus, we have that $G=PA$, and then $|G:A|$ is a power of $p$. On the other hand, $|G:QA|=|H:Q(A\cap H)|$ is not divisible by $p$, since $Q\in\syl p H$. This implies that $G=QA$, as wanted.

It remains to prove that $A$ is the unique minimal normal subgroup of $G$. Assume that $M$ is a minimal normal subgroup of $G$ different from $A$. Then $|M|$ divides $|G:A|$ and since $G/A$ is a $p$-group, we have that $M$ is an elementary abelian $p$-group. Reasoning as above, we would obtain that $G=MP=P$, which is impossible.\medskip

\textit{Step 6: $A=S_1\times S_2\times\cdots\times S_t$, where $S_i\cong S$ a nonabelian simple group of order divisible by $p$, for all $i=1,\ldots,t$, $\{S_1,\ldots, S_t\}$ is transitively permuted by $Q$ and $t$ is a power of $p$.}

\medskip

The group $G$ is not $p$-solvable by Theorem 5.1. 
Since $G/A$ is a $p$-group, $A = S_{1} \times \ldots \times S_{t}$ where $S_i \cong S$ is a nonabelian simple group of order divisible by $p$, for all $i = 1,\ldots t$, and the group $Q$ acts transitively on the set $\{S_1, \ldots , 
S_t\}$. In particular, $t$ is a power of $p$.
\end{proof}

\begin{rem}
Note that this shows that if $G$ does not have a section isomorphic to $C_p\wr C_p$ then a minimal counterexample is an almost simple group. The reduction to almost simple groups of the general case of Conjecture D appears to require different techniques to those in this paper. We plan to  address it and the almost simple case elsewhere.
\end{rem}

\section{Covering the set of $p$-elements}

Next we prove Theorem G. In fact, we prove a stronger statement.

\begin{thm}
\label{gstr}
Let $p$ be a prime. 
For any sequence of simple groups $G$ of order going to infinity and divisible by $p$, 
there exists a $p$-element $x \in G$ such that the minimal number of proper subgroups of $G$ that are necessary to cover the conjugacy class of $x$ goes to infinity. 
\end{thm}

\begin{proof}
It is sufficient to find a $p$-element $x$ in $G$ such that for any maximal subgroup $H$ in $G$ the proportion of elements in $x^{G}$ covered by $H$ goes to $0$ as the size of $G$ tends to infinity. This proportion is $|x^G\cap H|/|x^G|$, which is $\fpr(x,G/H)$ by \cite[Lemma 1.2(iii)]{bur}. The result follows from Theorem B.
\end{proof}

Finally, we prove Theorem F. This does not depend on fixed point ratios and does not use the classification of finite simple groups either.

\begin{proof}[Proof of Theorem F]
If $G$ is a $p$-group, the result is easy and well-known. Suppose that $G$ is not a $p$-group. 

Among all proper subgroups $H$ containing a Sylow $p$-subgroup choose one
with $H =\langle H_p\rangle$  and $|H_p|$ maximal. Let $P\in\Syl_p(G)$ such that $P\leq H$.

If $H\leq M$ with $M$ maximal, then $M_p=H_p$ and so $H$ is normal in $M$.  Since $G$ is generated by the conjugates of $P$, $M$ is not normal in $G$. Thus $H$ is in a unique maximal subgroup $M=\bN_G(H)$.

Arguing by induction, $H_p$ is not covered by the union of $p$ proper subgroups
and so
in any covering of $G_p$ by $p$ proper subgroups of $G$, one of them must contain
$H$.    Since we may assume our covering consists of maximal subgroups,
the unique maximal subgroup $M$ containing $H$ must be in the covering.
The same applies to any $H^g$ and so $M^g$ must also be in the covering as well.

Let $s =|G:M|$ be the number of conjugates of $M$.  Then $s \ne p$ since $M$ contains
$P$ and since $G =\langle G_p\rangle$, $ s \ge p$ (there are no homomorphisms from $G$ into
$S_{p-1}$ other than the trivial one).
We conclude that the covering has size $s>p$, as wanted.
\end{proof}

A variation of Lemma 2.3 of \cite{brg} can be used to prove that $\sigma_p(G)\geq 3$.

\section{Noncommuting subsets and graphs} \label{graph}

Given a finite group $G$, we consider the graph with vertices the elements of $G$ where two vertices $x,y\in G$ are joined by an edge if they do not commute. This is the so-called non-commuting graph of $G$, which has been very studied. By its definition, the clique number $n(G)$ of this graph   turns out to be the size of the largest subset of $G$ consisting of pairwise noncommuting elements. Answering a question of Erd\"os, it was proved by Neumann \cite{neu} that $|G:\bZ(G)|$ is bounded from above in terms of $n(G)$. Later, this bound was improved by Pyber to $|G:\bZ(G)|\leq c^{n(G)}$, for some constant $c$ in \cite{P}. This bound is of the right order of magnitude. 

In this section, we look for a local version of this result. Is it true that the index of the centre of a finite group $G$ is bounded from above in terms of the largest subset of pairwise noncommuting $p$-elements? We will write $n_p(G)$ to denote the size of this set. More generally, we will write $n_{\pi}(G)$ to denote the size of the largest subset of $G$ consisting of pairwise noncommuting $\pi$-elements. We will say that a subset $S$ of $G$ is noncommuting if for every $x,y\in S$, $x$ and $y$ do not commute. As a consequence of Theorem G, we have the following.

\begin{thm}
\label{simnc}
Let $p$ be a prime. Then $\sigma_p(G)\leq n_p(G)$. In particular $n_p(G)\to\infty$ when $G$ is a finite simple group of order divisible by $p$ and $|G|\to\infty$.
\end{thm}

\begin{proof}
Write $k=n_p(G)$ and let $\{x_1,\dots,x_k\}\subseteq G_p$ be a noncommuting subset of $G$ of maximal size. Then $G_p\subseteq\bigcup_{i=1}^k\bC_G(x_i)$, so $\sigma_p(G)\leq n_p(G)$. The result follows from Theorem G.
\end{proof}

Next, we relate our question with the commuting probability of $p$-elements.  We will use a graph theoretic argument.

All graphs considered here are finite,  undirected and have at most one edge between two vertices. Given a graph $\Gamma$ we write $\bar{\Gamma}$ to denote the complementary graph of $\Gamma$. We also  define the clique of $\Gamma$ as the size of the largest complete subgraph of $\Gamma$ (ignoring loops) and we denote it by $\omega(\Gamma)$. We will use  the following result of Tur\'an \cite{T}.

\begin{lem}[Tur\'an's theorem]\label{Turan}
Let $\Gamma$ be a loop-free graph with $n$ vertices. Then the number of edges in $\Gamma$ is bounded by

 $$(1-\frac{1}{\omega(\Gamma)})\frac{n^2}{2}.$$
\end{lem}

Now, we define $\Pro(\Gamma)$ as the proportion of edges in $\Gamma$, this is

$$\Pro(\Gamma)=\frac{|\{(x,y)\in V(\Gamma)\times V(\Gamma)|x\sim y\}|}{|V(\Gamma)|^2}.$$
Next result relates $\Pro(\Gamma)$ with the clique of its complementary graph in an special case.

\begin{lem}\label{generalgraph}
If $\Gamma$ is a graph containing all loops, then $\Pro(\Gamma) \omega(\bar{\Gamma})\geq 1$.
\begin{proof}
Since $\bar{\Gamma}$ possesses no loops and it is undirected, we have that the number of edges of $\bar{\Gamma}$ is

$$\frac{|\{(x,y)\in V(\Gamma)\times V(\Gamma)|x\not \sim y\}|}{2}=(1-\Pro(\Gamma))\frac{|V(\Gamma)|^2}{2}.$$

On the other hand, the largest clique in  $\bar{\Gamma}$ has size $ \omega(\bar{\Gamma})$. Thus, by Lemma \ref{Turan}, we have that

$$(1-\Pro(\Gamma))\frac{|V(\Gamma)|^2}{2}\leq (1- \omega(\bar{\Gamma})^{-1}   )\frac{|V(\Gamma)|^2}{2}.$$

The result follows.
\end{proof}
\end{lem}

Let $\pi$ be a set of primes and $G$ be a finite group. We define the commuting $\pi$-graph of $G$, which we denote as $\Gamma_{\pi}(G)$, as the graph whose vertices are the $\pi$-elements of $G$ and two $\pi$-elements are joined if they commute. We define the non-commuting $\pi$-graph of $G$ as the complementary of the graph $\Gamma_{\pi}(G)$, and we define the $\pi$-probability of $G$ as

$$\Pro_{\pi}(G)=\Pro(\Gamma_{\pi}(G))=\frac{|\{(x,y)\in G_{\pi}\times G_{\pi}| xy=yx\}|}{|G_{\pi}|^2},$$
where $G_{\pi}$ is the set of $\pi$-elements of $G$. Since every element commutes with itself, we have that $\Gamma_{\pi}(G)$ contains all loops. In addition, the largest clique in the non-commuting $\pi$-graph has size $n_{\pi}(G)$. Therefore,  applying Lemma \ref{generalgraph}, we have the following result.

\begin{lem}\label{equiv}
If $G$ is a finite group and $\pi$ is a set of primes, then $\Pro_{\pi}(G) n_{\pi}(G)\geq 1$.
\end{lem}

The following question was formulated by Burness, Guralnick, Moret\'o and Navarro at the end of \cite{bgmn}: is it true that $\Pr_p(G)\to0$ when $p$ is a prime, $G$ is a group generated by its $p$-elements, $\bO_p(G)=1$, $|G|\to\infty$? We show that a local version of Pyber's theorem is a consequence of this question.

\begin{thm}
\label{localneu}
Suppose that the above mentioned question has an affirmative answer. Then $|G|$ is bounded from above in terms of $n_p(G)$ when $p$ is a prime, $G$ is a group generated by $p$-elements with $\bO_p(G)=1$. 
\end{thm}

\begin{proof}
By Proposition \ref{equiv}, we have that 

$$n_p(G)\geq \frac{1}{\Pro_p(G)}$$
and by hypothesis $\frac{1}{\Pro_p(G)}$ tends to infinite as $|G|$ increases. Thus, $n_p(G)$ tends to infinite as $|G|$ increases and hence $|G|$ is bounded in terms of $n_p(G)$.
\end{proof}

It is worth remarking that another application of Lemma \ref{Turan} in this context appeared in the paragraph that follows Theorem 1.11 of \cite{gm}. 

\medskip

Let $G$ be a finite group and let $a(G), b(G),c(G),\ldots$ be invariants of $G$. As usual, we will say that $a(G)$ is bounded in terms of $b(G),c(G),\ldots$ (or that $a(G)$ is $(b(G),c(G),\ldots)$-bounded) if there exists a real valued function $h$ such that $a(G)\leq h(b(G),c(G),\ldots)$ for all finite groups.

We conclude this section with the proof of Theorem H. If $G$ is a 
finite group and $m, n$ are positive integers, we will say that $G\in C(m,n)$ if for every $S_1,S_2\subseteq G$ with $|S_1|=m, |S_2|=n$, there exist $x\in S_1, y\in S_2$ such that $xy=yx$. It was proved in Theorem 1.1 of \cite{aahz06} that there exists a function $f(m,n)$ such that if $G\in C(m,n)$ and $|G|>f(m,n)$ then $G$ is abelian. We prove the following local version, guaranteeing the existence of abelian Hall $\pi$-subgroups. If $\pi$ is a set of primes, we say that $G\in C(m,n)$ if for every $S_1,S_2\subseteq G_{\pi}$ with $|S_1|=m, |S_2|=n$, there exist $x\in S_1, y\in S_2$ such that $xy=yx$. 

Let $\pi$ be a set of primes. 
Recall that a group is $\pi$-separable if it has a normal series with factor groups that are $\pi$-groups of $\pi'$-groups. The $\pi$-separable radical of a finite group is the largest normal $\pi$-separable subgroup. Recall also  that $\pi$-separable groups have Hall $\pi$-subgroups (see Theorem 3.20 of \cite{isa}). We will also use Hall-Higman's Lemma 1.2.3 (Theorem 3.21 of \cite{isa}): if $G$ is $\pi$-separable and $\bO_{\pi'}(G)=1$, then $\bC_G(\bO_{\pi}(G))\leq\bO_{\pi}(G)$.

\begin{proof}[Proof of Theorem H]

1) We start with some elementary observations that we will use in the proof. Let $H\leq G$ and $N\trianglelefteq G$. We clearly have that $n_{\pi}(H)\leq n_{\pi}(G)$ and $n_{\pi}(G/N)\leq n_{\pi}(G)$. Furthermore, if $G\in C_{\pi}(m,n)$ then $H\in
C_{\pi}(m,n)$ and $G/N\in C_{\pi}(m,n)$. Note also that if $G\in C_{\pi}(m,n)$ then $n_{\pi}(G)<m+n$. Note also that if $p\in\pi$ then $n_p(G)\leq n_{\pi}(G)$.

2) Next, we suppose that $G$ possesses Hall $\pi$-subgroups and claim that if $|H|>f(m,n)$, where $H$ is a Hall $\pi$-subgroup of $G$ and $f$ is the function from Theorem 1.1 of \cite{aahz06}, then $H$ is abelian. Let $H$ be a Hall $\pi$-subgroup of $G$. We note that $H\in C_{\pi}(m,n)=C(m,n)$. By Theorem 1.1 of \cite{aahz06}, we have that if $|H|>f(m,n)$ then $H$ is abelian, as claimed.

\medskip

 Now, it suffices to show that if $G\in C_{\pi}(m,n)$ and  $|G|_{\pi}$ is large (not bounded in terms of $n$ and $m$), then $G$ possesses Hall $\pi$-subgroups. We will write $R$ to denote the $\pi$-separable radical of $G$. 

\medskip

3) Next, we claim that if $G\in C_{\pi}(m,n)$ has trivial $\pi$-separable radical then $|G|$ is $(m,n)$-bounded. Since $n_{\pi}(G)<m+n$, it sufficed to prove that $|G|$ is $n_{\pi}(G)$-bounded. By hypothesis, the generalized Fitting subgroup $\bF^*(G)$ of $G$ is a product of nonabelian simple groups, all of them of order divisible by some prime in $\pi$. Write $\bF^*(G)=S_1\times\cdots\times S_t$, where $S_i$ is simple nonabelian. By elementary properties of the generalized Fitting subgroup (Theorem 9.8 of \cite{isa}), we have that $G$ is isomorphic to a subgroup of $\Aut(\bF^*(G))$, so it suffices to prove that $|\bF^*(G)|$ is $n_{\pi}(G)$-bounded. For this, we will see first that the order of every (abelian simple) direct factor of $\bF^*(G)$ is $n_{\pi}(G)$-bounded. 
This follows from Theorem \ref{simnc} and the fact that $n_p(G)\leq n_{\pi}(G)$ for every $p\in\pi$. Now, it suffices to prove that $t$ is $n_{\pi}(G)$-bounded. We pick a noncommuting subset $\{a_i,b_i\}\subset (S_i)_{\pi}$ for every $i=1,\dots,t$. For $j=1,\dots, t$, we define 
$$
y_j=\prod_{i=1}^ja_i\prod_{i=j+1}^tb_i.
$$
We have that $\{y_1,\dots,y_t\}$ is a noncommuting subset of $(\bF^*(G))_{\pi}$. Thus $t\leq n_{\pi}(\bF^*(G))\leq n_{\pi}(G)$, as wanted.

4) Finally, let $R$ be the $\pi$-separable radical of $G$.  We suppose that $G\in C_{\pi}(m,n)$ does not have Hall $\pi$-subgroups and we want to prove that $|G|_{\pi}$ is $(m,n)$-bounded. 
By the previous paragraph, we know that $|G/R|$ is $(m,n)$-bounded, so it suffices to prove that $|R|_{\pi}$ is $(m,n)$-bounded. 

Suppose first that $R$ does not possess abelian Hall $\pi$-subgroup. Since $R\in C_{\pi}(m,n)$ and $R$ possesses Hall $\pi$-subgroup, we have by 2) that $|R|_{\pi}$ is $(m,n)$-bounded, as wanted. 

Assume now that $R$ possesses abelian Hall $\pi$-subgroups. We may assume that $\bO_{\pi'}(G)=1$. Now, Hall-Higman's Lemma 1.2.3 implies that  $K=\bO_{\pi}(G)$ is a normal Hall $\pi$-subgroup of $R$.

Since $G$ does not have Hall $\pi$-subgroups, we deduce that $G/K 
$ does not have Hall $\pi$-subgroups. Thus, there exist $x,y\in G_{\pi}$ such that $xK$ and $yK$ do not commute.  Therefore, for any $g\in xK$, $h\in yK$, $gh\neq hg$. It follows that $|K|\leq\max\{m,n\}$. This implies that $|R|_{\pi}=|K|$ is $(m,n)$-bounded, as wanted.  
\end{proof}

\section{Examples and further questions}

We conclude with some examples showing that our results are close to best possible. First, if $G$ is any almost simple group with socle $S$ such that $G/S\cong C_2\times C_2$, then we have that $G$ is covered by three maximal subgroups $H_1, H_2$ and $H_3$. Thus, if $x \in G_2$, then $x^G\subseteq H_1\cap H_2 \cap H_3$ and hence there exists some $H_i$ containing at least $1/3$ of the elements in $x^G$, or equivalently $\fpr(x,G/H_i)\geq 1/3$. Since, there are infinitely many  simple groups whose outer automorphism group contains a subgroup isomorphic to $C_2\times C_2$, This proves that Theorems B and G do not hold for almost simple groups.

It is easy to see that if $p\geq3$ is a prime, $G=A_{2p-1}$ and $H=A_{2p-2}$, then $\nu_p(H)/\nu_p(G)=(p-1)/(2p-1)$. Thus, the first bound in Theorem  C is best possible. Since in these examples $|G|_p=p$ and, by Lemma \ref{SylowFPR}, this quotient of Sylow numbers is a fixed point ratio, the first bound of Theorem A is also sharp.


Regarding Conjecture D, for all groups $G$ in the PerfectGroups library in GAP \cite{GAP} and every subgroup $H<G$ with $\nu_p(H) < \nu_p(G)$, the inequality $$\nu_p(H) \leq \frac{p-1}{2p-1} \nu_p(G)$$ holds when $p>3$. For all the counterexamples $(G,H)$ for $p=3$ that we found, we have $\nu_p(H) = (4/7) \cdot \nu_p(G)$. This suggests that in Conjecture D perhaps $f(3)$ may be taken to be $4/7$. The only simple groups in the library where this equality is attained are $\mathrm{PSU}_{3}(3)$ and $\mathrm{SL}_{3}(4)$. This is consistent with the possibility that Theorem E holds for odd primes without the hypothesis that $H$ contains a Sylow $p$-subgroup of $G$. 

Next, we present both solvable and non-solvable examples showing that Theorem E is false even for groups generated by a conjugacy class of $p$-elements. Let $n\geq 3$ be an integer. Let $G_1=S_n$ and for every $m\geq 2$ define $G_m=G_{m-1}\wr S_n$ Let $H< S_n$ such that $H$ contains a Sylow $2$-subgroup of $S_n$ and $H_m=G_{m-1}\wr H$. Then 
$$\frac{\nu_2(H_m)}{\nu_2(G_m)}=\frac{\nu_2(H)}{\nu_2(S_5)}$$
for all $m$.

Furthermore, by Lemma 2.1(i) of \cite{mmm}, $\sigma_2(G_m)\leq\sigma_2(G_m/(G_{m-1})^n)=\sigma_2(S_n)$ does not go to infinity when $m\to\infty$. Therefore, Theorem G is also false for groups generated by a conjugacy class of $p$-elements.

We have seen in Theorem F that if $G$ is a group generated by $p$-elements, then $\sigma_p(G)\geq p+1$. In view of Theorem \ref{gstr}, it seems natural to ask the following. Let $G$ be a simple group and let $x\in G$ be a $p$-element. Is it true that the minimal size of a covering of $x^G$ is at least $p+1$? Perhaps, this could be true even for not necessarily simple groups with $G=\langle x^G\rangle$.

\subsection*{Acknowledgments}  Parts of this work were done when the second author was visiting the University of Valencia and the fourth and fifth authors were visiting the Hun-Ren Alfr\'ed R\'enyi Institute of Mathematics. They thank the corresponding research teams for their hospitality.

\end{document}